\documentclass{amsart}
\usepackage{amssymb}
\usepackage{amsmath}
\usepackage{color}
\usepackage{graphicx}
\usepackage{dsfont}

\newtheorem{theo}{\indent Theorem}[section]
\newtheorem{prop}[theo]{\indent Proposition}

\newtheorem{rem}[theo]{\indent Remark}

\newtheorem{cor}[theo]{\indent Corollary}

\newtheorem{ex}[theo]{\indent Example}
\newtheorem{ass}[theo]{\indent Assumption}

\def \XX{\bar X^{G }}
\def\ttt{\mathbf{t}}
\def\zz{\mathbf{z}}
\def\hh{\mathbf{h}}
\newcommand{\R}{\mathbb {R}}
\newcommand{\N}{\mathbb {N}}

\newlength{\breite}
\title[Convergence to equilibrium for time inhomogeneous processes]{Convergence to equilibrium for time inhomogeneous jump diffusions with state dependent jump intensity}
\date{October 29, 2018}
\author{E. L\"ocherbach }

\address{E. L\"ocherbach: SAMM, Universit\'e de Paris 1 Pant\'ehon Sorbonne,
90 rue de Tolbiac, 75013 Paris, France.}

\email{eva.locherbach@univ-paris1.fr}

\begin{document}
\maketitle
\def\abstractname{Abstract}
\begin{abstract}
We consider a time inhomogeneous Markov process $X = (X_t)_t$ with jumps having state dependent jump intensity, with values in $\R^d , $ and we are interested in its long time behavor. The infinitesimal generator of the process is given for any sufficiently smooth test function $f$ by 
$$
L_t f (x) = \sum_{i=1}^d \frac{\partial f}{\partial x_i } (x) b^i ( t,x)  
+ \int_{\R^m } [ f ( x + c  ( t, z, x)) - f(x)] \gamma  ( t, z, x) \mu  (dz )  ,
$$ 
where $ \mu   $ is a sigma-finite measure on $(\R^m , {\mathcal B} ( \R^m ) ) $ describing the jumps of the process.

We give conditions on the coefficients $ b(t, x) ,  c(t, z, x) $ and $ \gamma ( t, z, x ) $ under which the long time behavior of  $X$ can be related to the one of a time homogeneous limit process $\bar X . $ Moreover, we introduce a coupling method for the limit process which is entirely based on certain of its big jumps and which relies on the regeneration method. We state explicit conditions in terms of the coefficients of the process allowing to control the speed of convergence to equilibrium both for $X$ and for $\bar X.$ 

\end{abstract}

\\

{\it MSC 2000}  : 60 J 55, 60 J 35, 60 F 10, 62 M 05

\section{Introduction}
In this paper we study a rather general class of jump type stochastic differential equations taking values in $ \R^d  $ evolving according to  
\begin{equation}\label{eq:sde}
X_t = x +\int_0^t b (s, X_s) d s + \int_{[0, t ]} \int_{  \R^m \times \R_+} c (s, z, X_{s-}) 1_{ u \le \gamma  ( s,z, X_{s-})}  N (ds,d z, du ) ,
\end{equation}
with $x \in \R^d . $  In the above equation, $N(ds, dz, du ) $ is a Poisson random measure, defined on a fixed probability space $(\Omega, {\mathcal A}, P) ,$ with $ (s, z, u ) \in \R_+ \times \R^m  \times \R_+, $ having intensity measure $ d s \mu  (d z)  d u ,$ for some $\sigma -$finite measure  $\mu $ on $ (\R^m , {\mathcal B} ( \R^m )   ) .$ The associated infinitesimal generator at time $t$ is given by  
\begin{equation}\label{eq:gen}
L_t f (x) = \sum_{i=1}^d \frac{\partial f}{\partial x_i } (x) b^i ( t,x) 
+ \int_{\R^d} [ f ( x + c ( t, z, x)) - f(x)] \gamma ( t, z, x) \mu (dz )  .
\end{equation} 

The coefficients of the system are the measurable functions $ b : \R_+ \times \R^d \to \R^d , $ $c : \R_+  \times \R^m \times \R_+  \to \R^d $ and $ \gamma : \R_+ \times \R^m  \times \R_+  \to \R_+  .$ We shall always work under conditions ensuring that \eqref{eq:sde} admits a unique strong non-explosive adaptive (to the filtration generated by the Poisson random measure) solution which is Markov, having c\`adl\`ag trajectories which are of finite variation (see Assumption \ref{conditions} below). 

Let us give some comments on \eqref{eq:sde}. If the jump rate $ \gamma $ is a constant function, then the above process is a classical jump process. But if $\gamma $ is not constant, then the jump intensity and also the jump amplitude depend on the current position of the process. This is a natural assumption in many modeling issues (see e.g. \cite{CDMR}, \cite{PTW-10} or \cite{pierrennathalie} for the modeling of biological or chemical phenomena, see \cite{ABGKZ} for an overview). 

Observe that if the measure $ \mu$ is finite, then the process evolving according to \eqref{eq:sde} is a piecewise deterministic Markov process (PDMP). PDMP's have been introduced by Davis (\cite{Davis84} and \cite{Davis93}); they evolve in a deterministic manner in between successive jump events, and only a finite number of jumps occur during finite time intervals. Here, we will however deal with the general infinite activity case. 

The goal of this paper is to describe how the solution of \eqref{eq:sde} behaves as $t \to \infty .$ 
Let us illustrate the main ideas by some examples. Suppose e.g.\ that there exist measurable functions $ c ( z, x) : \R^m \times \R^d \to \R^d  ,$ $ \gamma ( z, x) : \R^m \times \R^d \to \R_+  $ and $  b_1 ( x) : \R^d \to \R^d $ such that for all $ z \in \R^m , x\in \R^d , $ 
$$ |c(t, z, x) - c(z, x) | + | \gamma ( t, z, x ) - \gamma ( z, x) | +  | b(t, x) - b_1 ( x) |   \to 0 \mbox{  as } t \to \infty .$$ 

Then (under suitable additional technical conditions) the long time behavior of \eqref{eq:sde} will be well-described by another process of the same type,  solution of 
$$ 
\bar X_t = x +\int_0^t b ( \bar X_s) d s + \int_{[0, t ]} \int_{  \R^m \times \R_+} c (z, \bar X_{s-}) 1_{ u \le \gamma  ( z,\bar  X_{s-})}  N (ds,d z, du ) ,
$$
having {\it time homogenous coefficients.} We will call this regime of convergence the {\it slow regime}, and we notice that in this case jumps survive in the limit process. 

Of course this is not the only possible scenario, and two other limit regimes exists. They  both appear in the situation where, as $t \to \infty, $ the jump heights tend to $0.$ If they do so in a moderate way, they will just produce a limit drift. If they converge to $0$ sufficiently fast, they will generate a limit diffusive part. 

Suppose e.g. that the {\it drift} produced by the jumps at time $t, $ given by 
$$ \tilde b( t, x) = \int_{\R^m} c (t, z, x) \gamma ( t, z, x ) \mu ( dz), $$
converges to a limit drift vector field $b_2 ( x) $  and that $\int_{\R^m} |c (t, z, x)|^2  \gamma ( t, z, x ) \mu ( dz) \to 0 $ as $t \to \infty .$ In this case, the corresponding time homogenous limit process $\bar X_t$  is solution of the deterministic equation 
$$ d \bar X_t = b (\bar X_t ) dt , \; \mbox{ with } b (x) = b_1 (x) + b_2 (x) .$$ 
We call such a limit regime the {\it intermediate jump regime} - it produces a deterministic limit process, and such results are of course related to the law of large numbers. 

Finally, the jump part in \eqref{eq:sde} can be centered, that is, $ \tilde b( t, x) = 0 $ for all $ t, x .$ In this case, we are in the {\it fast jump regime} -- and interesting limit features may appear if the variance of the jump part given by 
$$   a^{ij} ( t, x) = \int_{\R^m} c^i (t, z, x) c^j (t, z, x) \gamma ( t, z, x ) \mu  (dz ) , 1 \le i, j \le d ,$$
converges, as $t \to \infty , $ to some limit variance $ a(x) $ giving rise to a limit diffusive part, and if at the same time, higher order terms $ \int_{\R^m} |c (t, z, x)|^3  \gamma ( t, z, x ) \mu ( dz)  $ disappear, as time tends to infinity. 

Notice that these three jump regimes may appear simultaneously as shows the following example.

\begin{ex}\label{ex:CIR}
Let $ d=m=  1 $ and $ \mu ( dz ) = dz . $ For $t > 0$ and for some $r >0,$ let 
$$ c (t, z, x ) = 
\frac{\sigma }{2 } e^{-rt}  [ 1_{ ]-3 e^{2rt}  , -2e^{2rt} [ } (z) - 1_{ ] -2e^{2rt} , -e^{2rt}  [ } (z) ]  - a x e^{-2rt}   1_{ ]- e^{2rt}, 0 [ } (z) + \frac{d}{(1 + z)^2 }  1_{ ]0, e^{2rt} [ } (z),$$
$b(t, x ) = b $  and 
$$ \gamma(t, z, x) =  f(x) 1_{ ]-3 e^{2rt}  , - e^{2rt} [ } (z) +   1_{ ]- e^{2rt}, 0 [ } (z)  +   f(x) 1_{ ]0, e^{2rt} [ } (z) , $$
for some constants $ \sigma , d \in \R $ and $ a, b  > 0 .$ Here, $f : \R \to \R_+$ is bounded, having bounded derivative such that $ \inf_{ x \in \R } f( x) > 0 .$ Let $X_t$ be solution of \eqref{eq:sde} with these parameters, starting from $X_0 = x \in \R_+  . $ 

By construction, jumps coming from ``noise'' $ z \in ]-3  e^{2rt},- e^{2rt} [ $ are centered, that is, 
$$ \int_{-3 e^{2rt}}^{-e^{2rt}} c(t, z, x) \gamma ( t, z , x ) \mu (dx ) = 0, $$ 
and the associated variance is given by $ a (t, x) = \sigma^2   f(x) $ for all $ t, x .$ 
Moreover, for all $ t, x, $ 
$$ \int_{-e^{2rt}}^{0} c(t, z, x) \gamma ( t, z , x ) \mu (dx ) =  - a x $$ 
giving rise to a limit drift $-a x  .$ Finally, the jumps produced by noise $z > 0 $ survive in the limit process -- this corresponds to the slow jump regime. It is straightforward to show that the associated limit process is a Cox-Ingersoll-Ross type jump process given by  
$$
d \bar X_t =  (b - a  \bar X_t    ) dt + \sigma  \sqrt{  f (\bar X_t ) } d W_t  +  \int_{ \R  \times \R_+} \frac{d}{(1 + z)^2 }  1_{ \{ u  \le  f( \bar X_{t-}) \}}  N (dt,d z, du ) .
$$
\end{ex}

Let us now come back to the general frame of \eqref{eq:sde}. In this paper, we will given conditions on the coefficients of the system that guarantee that the associated limit process is a time homogeneous jump diffusion process $\bar X_t $ having generator
\begin{multline}\label{eq:12}
L f(x) = \sum_{i=1}^d \frac{\partial f}{\partial x_i } (x) g^i  ( x) + \frac{1}{2} \sum_{i, j = 1}^d \frac{\partial^2 f}{\partial x_i \partial x_j  } (x) a^{i j }  ( x)
+ \int_{{\R^m}} [ f ( x + c (  z, x)) - f(x)] \gamma (  z, x) \mu (dz ),
\end{multline}
where $ g = b_1 + b_2, $ with $b_2$ the limit drift of the intermediate regime, $a  $ the limit variance associated to the fast regime  and $c $ and $\gamma $ the limit jump intensity and height of the slow regime. 

Let us now define precisely what we mean when saying that $ \bar X$ is the limit process associated to \eqref{eq:sde}. We formalize this idea by using the notion of {\it asymptotic pseudotrajectories} which has been introduced in Bena\"{\i}m and Hirsch  \cite{bh96} (1996) and then further used in \cite{bouguetetal} and which provides a general framework to deal with the long time behavior of non-autonomous processes. 

For $X_t $ a solution of \eqref{eq:sde}, starting from any arbitrary initial law $ {\mathcal L} ( X_0) , $ let 
$$ \mu_t := {\mathcal L} (X_t) $$
and write $P_t$ for the transition semigroup of the limit process $\bar X_t.$ We introduce for any two probability measures $ \mu $ and $\nu $ on $(\R^d , {\mathcal B} ( \R^d ) ) $ and any class of test functions $ {\mathcal F}$ the distance
$$ d_{\mathcal F} ( \mu , \nu ) := \sup_{ f \in {\mathcal F}} | \mu ( f) - \nu ( f) | .$$
The main result of this paper, Theorem \ref{cor:1}, gives explicit conditions on the coefficients of \eqref{eq:sde} implying that for a suitable class of test functions $ {\mathcal F} $ and for any $ T < \infty , $ 
\begin{equation}\label{eq:aspseudo}
 \lim_{t \to \infty } \sup_{s \le T} d_{\mathcal F} ( \mu_{t + s } , \mu_t P_s ) = 0.
\end{equation}
This means that $ ( \mu_t)_t $ is an asymptotic pseudotrajectory of $ (P_t)_t $ in the sense of \cite{bh96}. 
Furthermore, if the limit process $\bar X_t$ is exponentially ergodic with unique  invariant probability measure $\pi , $ we shall also prove the weak convergence of $\mu_t$ to $  \pi , $ as $ t \to \infty ,$ together with a control of the speed of convergence (see Corollary \ref{cor:3}). 

The study of non-stationarity is a challenging problem, in particular in statistics of stochastic processes, and a lot of papers have been devoted to this topic during the last decade, see e.g. \cite{JM}, \cite{Dahlhaus2017} or \cite{hlt}  and the references therein. All these papers deal either with the time discrete case or with some periodic setting (leading again to a discrete description), and to the best of our knowledge, the results of Theorem \ref{cor:1} and Corollary \ref{cor:3} are one of the first results that allow to deal  with the longtime behavior of time inhomogeneous processes in the framework of continuous time observations.

To prove \eqref{eq:aspseudo}, we have to control two schemes of convergence. Firstly the convergence of $X_t $ to $ \bar X_t ,$ and secondly, the convergence of $ \bar X_t$ to its invariant regime $ \pi .$ 

The main point to prove the convergence of $X_t$ to the limit process is a control of the asymptotic behavior of the generator $L_t $ as $t \to \infty , $ using a Taylor expansion, and to prove that $ L_t $ converges to $L, $ in a sense that has to be made precise. It is then a classical approach, relying on the Trotter-Kato theorem, to show that this convergence implies the one of the associated transition semigroups. This implication is classical if the limit transition semigroup {\it preserves regularity}, but it is more difficult in the present frame, since the presence of the position dependent jump rate $ \gamma $ makes the regularity analysis of the associated transition semigroup more intricate and difficult. We rely on recent results of Bally et al. \cite{ballyetc} (2017) which enable us to solve this difficulty.

Let us now comment on the second scheme of convergence, the convergence of $\bar X_t$ to its invariant measure $ \pi . $ For classical jump diffusions there starts to be a huge literature on this subject. Masuda \cite{Ma07} (2007) follows the Meyn and Tweedie approach developed in \cite{Me93book} or \cite{Me93}, but he works in the simpler situation where the  intensity term $\gamma ( z, x) $ of \eqref{eq:12} is not present. Kulik \cite{kulik} (2009) uses the stratification method to prove exponential ergodicity of jump diffusions, but the models he considers do not include position dependent jump rates neither. Finally, Duan and Qiao \cite{DQ14} (2016) are interested in solutions driven by non-Lipschitz coefficients. 
None of the above mentioned papers is applicable to our situation due to the presence of the position dependent jump intensity $ \gamma ( z, x ) ,$ and therefore the first part of this paper is devoted to the ergodicity analysis of the process $ \bar X_t.$  More precisely, we show that the jumps themselves can be used in order to generate an explicit coupling method which leads to a control of the speed of convergence to equilibrium of $ \bar X_t.$ This coupling method  relies on the regeneration method which has been introduced in L\"ocherbach and Rabiet \cite{Victor} (2017) and which is applied to the big jumps. 

In spirit of the splitting technique introduced by Nummelin \cite{Nu78} (1978) and Athreya and Ney \cite{At78} (1978), we state a non-degeneracy condition which guarantees that the jump operator associated to the big jumps possesses a Lebesgue absolutely continuous component. This amounts to imposing that the partial derivatives of the jump term $c (z, x) $ in \eqref{eq:12} with respect to $z$ are sufficiently non-degenerate, see Assumption \ref{ass:final} below, leading to a sort of local Doeblin condition for the jump operator associated to big jumps. Roughly speaking this local Doeblin condition implies that  jumps issued of a pre-jump position $x $ belonging to a certain set $ C$ generate noise independently of the starting position $x.$ This relevant set $C$ would be what \cite{Me93book} call a ``petite set''. 

In order to be able to couple two trajectories of the limit process, we have then to ensure that they happen to visit the set $C$ at the same time. This is granted by a Lyapunov type condition implying that the process returns to a big compact $K$ infinitely often, together with a control of the moments of the associated hitting times. Moreover, we need a control argument that allows to steer the trajectory of $\bar X_t $ from any starting position $x \in K$ to the target set $ C$-- this control is based on an approximation of the process $\bar X$ by a finite activity process  where only big jumps are considered, see Theorem \ref{theo:controlXbar} below.  As a consequence, we are able to state our Theorem \ref{theo:tv} which proves the unique ergodicity of $\bar X$ together with a control of the speed of convergence to equilibrium with respect to total variation distance.

Our paper is organized as follows. Section \ref{sec:1} is devoted to the proof of the unique ergodicity of the limit process $ \bar X_t$ together with the control of the speed of convergence to equilibrium, see Theorem \ref{theo:tv}. In this section, we also explain the regeneration technique based on big jumps and prove the existence of a finite coupling time of two copies $ \bar X_t^x $ and $ \bar X_t^y , $ together with the existence of all of its polynomial moments under a suitable Lyapunov type condition. Section \ref{sec:2} is devoted to the study of the time inhomogeneous process $ X_t $ having generator \eqref{eq:gen}. Here, we state  Theorem \ref{cor:1} which proves the convergence of $ X_t$ to the limit process $\bar X_t , $ and Theorem \ref{theo:2} together with Corollary \ref{cor:3} proving the convergence of ${\mathcal L}(X_t) $ to the unique invariant probability measure $ \pi $ of the limit process $ \bar X_t .$ Finally, Section 4 gives some examples, in particular we discuss systems of Hawkes processes with mean field interactions, exponential memory kernels and variable length memory. Some mathematical proofs are shifted to the Appendix section.

\section{Unique ergodicity and Speed of convergence to equilibrium for jump diffusions with state dependent jump intensity}\label{sec:1}
\subsection{Notation}
Throughout this paper, for $x \in \R^d, $ $| x|  $ will denote the Euclidean norm on $\R^d .$ 
Moreover, for a multi-index $\alpha = ( \alpha_1, \ldots, \alpha_q) \in \{ 1 \ldots, d \}^q , $ we write $ | \alpha | = q $ for the length of $\alpha $ and $ \partial_x^\alpha = \partial_{x_{\alpha_1}} \ldots \partial_{x_{\alpha_q}}$ for the associated partial derivative. For a function $ f : \R^d \to \R $ which is $q$ times differentiable, we introduce for any $ 0 \le l \le q , $ 
\begin{equation} \| f \|_{l, q, \infty } :=  \sum_{l \le  | \alpha | \le q } \sup_{x \in \R^d } | \partial^\alpha_x f (x) | \mbox{ and } \| f \|_{q, \infty } := \| f \|_\infty + \| f \|_{1, q , \infty } .
\end{equation}
We write $ C^q_b ( \R^d ) $ for the class of all functions $f : \R^d \to \R$ such that $ \| f \|_{q, \infty } < \infty .$

\subsection{The model}

Let $\mu$ be a  $\sigma-$finite measure  on $({\R^m} , {\mathcal B} ({\R^m})).$ Moreover, let $(\Omega, {\mathcal A} , P ) $ be a probability space on which are defined a Poisson random measure $ N ( ds, dz , du ) ,$ which is a measure on $\R_+ \times {\R^m} \times \R_+$ having intensity measure $ ds \mu ( dz) du , $  as well as independent, standard $1-$dimensional Brownian motions $W^1, \ldots , W^k $ which are independent of $ N .$ We consider the following jump diffusion equation taking values in $\R^d, $ 
\begin{multline}\label{eq:sde0}
\bar X_t = x +\int_0^t  g   ( \bar X_s) d s +  \sum_{l=1}^k \int_0^t  \sigma_l ( \bar X_s) d  W^l _s \\
+ \int_{[0, t ]} \int_{  {\R^m} \times \R_+}  c ( z, \bar X_{s-}) {1}_{ u \le  \gamma ( z, \bar X_{s-})} N  (ds, d z, d u ) ,
\end{multline}
where the coefficients $ g (x),   $ $ c(z, x)$ and $ \gamma (z, x) $ are measurable functions satisfying the following assumption.

\begin{ass}\label{conditionsbis}
\begin{enumerate}
\item
$g$ and $ \sigma_l, 1 \le l \le k, $ are $C^1 $ and $ \| \nabla g \|_\infty  +\sum_{l=1}^k \| \nabla \sigma_l  \|_\infty < \infty .$ 
\item
$c$ and $ \gamma$ are continuous, Lipschitz continuous with respect to $x,$ i.e. 
$$ |c ( z, x ) - c (z, x') | \le L_c( z ) | x - x ' | \mbox{ and } |\gamma ( z, x ) - \gamma (z, x') | \le L_{\gamma}( z ) | x - x ' |,$$
where $ L_c, \ L_{\gamma} $ are measurable functions $  {\R^m}  \to \R_+ .$ 
\item

\begin{equation}\label{eq:cmu}
C_\mu ( \gamma , c ) :=  \sup_{x \in \R^d  } \int_{{\R^m}  }( L_{c} ( z)  \gamma (z,x) +  L_{\gamma} ( z) |c (z,x) |) \mu  (d z)  < \infty .
\end{equation}
\item
$\sup_{x \in \R^d }  \sup_{z \in {\R^m}}   \gamma (z,x)  = \Gamma < \infty  .$
\item $\sup_{ x \in \R^d } \int_{G} | c (z, x) | \gamma ( z, x) \mu (dz) < \infty .$ 
\end{enumerate}
\end{ass} 

Under these assumptions, Theorem 1.2.\ of Graham  \cite{carl} (1992) implies that  \eqref{eq:sde0} admits a unique strong non-explosive adapted solution which is Markov, having c\`adl\`ag trajectories.
We denote by $ \bar X_{t_0, t}^x , t \geq t_0, $  a version of the above process starting from the position $x \in \R^d $ at time $t_0 .$ Whenever $t_0 = 0, $ we shall shortly write $ \bar X_t^x := \bar X_{0, t }^x .$ Finally, if we do not want to specify the initial value at time $0$ of the process, we shall simply write $ \bar X_t $ for the process. 

Our aim is to state easily verifiable conditions on the parameters $ g, \sigma_l, c $ and $ \gamma $ of the process that imply the unique ergodicity of the process $ \bar X$ together with a control of the speed of convergence to equilibrium. Our approach relies on the regeneration technique based on the jump transitions. Therefore, we first state a sufficient condition implying a local Doeblin condition for the jump transitions. This is done in the next subsection and used ideas  introduced in \cite{Victor}.

\subsection{A Doeblin lower bound for the jump-transitions}
We control the rate of convergence to equilibrium based on a splitting scheme reminiscent of the regeneration technique. This scheme is entirely based on certain {\it big} jumps of $\bar X$ (that will be defined below) and needs the following non-degeneracy condition on the jump noise and the associated jump rate.
\begin{ass}\label{conditions2}
We suppose that $m \geq d.$ Let $ \mu  = \mu_{ac} + \mu_{s}  $ be the Lebesgue decomposition of $\mu ,$ with $ \mu_{ ac} (d z) = h(z) d z ,$ for some measurable function $h \geq 0 \in L^1_{loc} ( \lambda ) ,   $ where $\lambda $ is the Lebesgue measure on $\R^m .$ 

Then  there exist $x_0 \in \R^d ,   z_0 \in \R^m $ and $r,  R > 0 $ such that 
$$ \inf_{ z : |z - z_0 | \le R  , x : |x - x_0|\le r  } \gamma (z,x) h (z)  = \varepsilon  > 0. $$ 
 \end{ass}

In order to introduce what we shall call {\it big jumps} of the process, we impose the following condition which implies that the measure $ \gamma (x, z) \mu ( dz)  $ is sigma-finite, uniformly in $x.$ 

\begin{ass}\label{conditions2bis}
There exists a non-decreasing sequence $(G_n)_n $ of subsets of $ \R^m  $ and an increasing sequence of positive numbers $ \Gamma_n $ with $ \Gamma_n \uparrow +\infty $ as $n\to \infty, $ such that $ \bigcup G_n = \R^m ,$
\begin{equation}\label{eq:explosion}
 \int_{G_n}\gamma(z, x) \mu (d z ) =: \bar \gamma_n ( x) \le  \Gamma_n < \infty 
\end{equation}
for all $x \in \R^d $ and for all $n.$  
\end{ass}

We fix some $n.$ Thanks to \eqref{eq:explosion}, we can couple the process $\bar X_t$ with a rate $\Gamma_n-$Poisson process $ N^{[n]} = (N^{[n]}_t)_{t \geq 0}$ such that jumps of $\bar X_t$ produced by noise $z \in G_n,$ 
$$ \Delta^{[n]}  \bar X_t : = \int_{G_n } \int_0^\infty c( z, \bar X_{t-}) 1_{ \{u \le \gamma ( z, \bar X_{t-})\} } N (dt, dz, du) ,$$
can only occur at the jump times $ T_k^{[n]} , k \geq 1 ,$ of $N^{[n]}.$ We will construct our regeneration scheme based on these {\it big jumps} $ T_k^{[n]}, k \geq 1, $ for a suitably chosen truncation level $n.$ Let 
\begin{equation}\label{eq:Pi}
 \Pi ^{[n]} ( x, dy ) = {\mathcal L} ( \bar X_{T_k^{[n]} } | \bar X_{T_k^{[n]} - } = x ) (dy )
\end{equation}
be the transition kernel associated to big jumps. Let $ x_0, z_0 \in \R^d $ be as in Assumption \ref{conditions2}. We then assume

\begin{ass}\label{ass:final}
$ \nabla_z c ( z_0, x_0 ) $ has full rank. 
\end{ass}

\begin{prop}
Grant Assumptions \ref{conditions2}, \ref{conditions2bis} and \ref{ass:final}. Let $ r , R > 0 $ be as in Assumption \ref{conditions2} and fix $n_0 $ such that  $ \{ z \in \R^m : | z - z_0 | \le R \} \subset G_{n_0 }.$ Then there exist $\eta > 0, \beta \in ]0 , 1 [ $ and probability measure $ \nu $ on $(\R^d, {\mathcal B}(\R^d ) ) $ such that for any  $ V \in {\mathcal B} ( \R^d ) , $ 
\begin{equation}\label{eq:minoration}
\inf_{ x : |x - x_0| <  \eta } \Pi^{[n]} ( x, V) \geq  \beta \nu ( V) .
\end{equation} 
\end{prop}

\begin{proof}
Write $ \mathcal{K} = \{ z\in \R^m : | z - z_0 | \le R \}  $ with $z_0 $ and $ R$ chosen according to Assumption \ref{conditions2}. Then for all $ V \in {\mathcal B} ( \R^d ) $ and for all $ x , |x- x_0 | \le r , $  
\begin{multline}
\Pi^{[n]}( x, V) \geq  \frac{1}{\Gamma_n } \int_{G_{n}\cap {\mathcal K}}   \gamma ( z, x ) \mathds{1}_V ( x + c ( z, x ) )\mu (d  z) \\
\geq  \frac{1}{\Gamma_n} \int_{G_{n}\cap {\mathcal K}}   \gamma ( z, x ) \mathds{1}_V ( x + c ( z, x ) )h(z) d z 
\geq \frac{\varepsilon}{\Gamma_n} \int_{G_{n}\cap {\mathcal K}}   \mathds{1}_V ( x + c ( z, x ) ) d z, 
\end{multline}
where $h$ is the Lebesgue density of the absolute continuous part of $\mu .$ Since $ \nabla_z c ( z_0, x_0) $ is of full rank, standard arguments imply that the mapping $ z \mapsto x + c(z, x) $ is locally invertible, locally around $z_0, $ uniformly in $ x $ belonging to a small ball around $x_0,$ and the assertion follows e.g.\ from Lemma 6.3 of \cite{micheletal}. 
\end{proof}

\subsection{Total variation coupling}
We now explain how the  lower bound \eqref{eq:minoration} allows to couple two trajectories $\bar X_t^x $ and $ \bar X_t^y , $ the first issued from $x, $ the second from $y $ at time $t= 0,$ once they have both entered $ C := \{ x \in \R^d  : |x- x_0 | < \eta \} .$ These ideas rely on the regeneration technique introduced by Athreya and Ney  \cite{At78} (1978) and by Nummelin \cite{Nu78} (1978). We apply these ideas here to the jump mechanism. 

Firstly, fix some $n \geq n_0 $ and write $\Pi ( x, dy ) := \Pi^{[n]} (x, dy) $ for the jump transition kernel of \eqref{eq:Pi}. For any fixed $ x, x' \in \R^d , $ let $ \Pi ((x, x'), dy dy') $ be the maximal coupling of $ \Pi ( x, dy ) $ and $ \Pi ( x', dy') .$ Then \eqref{eq:minoration} implies that 
$$ \Pi ((x, x'), dy dy')  \geq \beta 1_{C \times C} (x, x') \nu (dy ) \delta_{y} (dy'). $$

This lower bound implies that once a jump $T_k^{[n]} $ occurs while the two copies of the process are inside the set $C,$ {\it with probability $ \beta, $ is is possible to choose the same  ``after-jump'' position $y$ for the two of them according to the measure $ \nu .$ }

We may therefore introduce a split kernel $Q ((x,x', u),dy dy' ), $ which is a transition kernel from $ \R^d \times \R^d \times [0, 1 ] $ to $\R^d  \times \R^d, $ defined by 
\begin{equation}\label{Q}
Q((x,x',u), dy dy') = \left\{
\begin{array}{ll}
\nu(dy) \delta_{y } (dy')  & \mbox{ if } (x,x',u) \in C \times C  \times [0, \beta]\\
\frac{1}{1 - \beta} \left( \Pi  ((x,x') , dy dy' ) - \beta \nu(dy) \delta_{y} (dy') \right)  & \mbox{ if } (x,x',u) \in C \times C \times ] \beta , 1] \\
\Pi  ((x, x') ,dy dy')  & \mbox{ if } (x,x')  \notin C \times C .
\end{array} \right. 
\end{equation}
Notice that 
$$ \int_0^1 Q((x,x', u), dy dy' )  du = \Pi (( x,x'), dy dy' ) ;$$
in this sense  $Q((x,x', u), dy dy' )$ can be considered as `splitting' the original  kernel $\Pi (( x,x'), dy dy' )$ by means of the additional `color' $u.$ 

We now show how to construct a coupled version of the processes $\bar X^x$ and $\bar X^y$  recursively over time
intervals $[T_k^{[n]} ,  T^{[n]}_{k+1} [ , k \geq 0 .$ For that sake introduce the process $Z^x_t$ defined by 
\begin{equation}\label{eq:z}
Z^x_t = x + \int_0^t g ( Z^x_s ) d s +\sum_{l=1}^m  \int_0^t \sigma_l ( Z^x_s) d  W^l _s + \int_0^t \int_{ G_n^c} \int_0^ { \infty } c( z, Z^x_{s-} ) 1_{ u \le \gamma ( z, Z^x_{s- }) } N (d s, d z, d u ) .
\end{equation}

The coupling construction works as follows.
\begin{enumerate}
\item
We use the same realization of jump times $ T_k^{[n]}, k \geq 0, $ for $\bar X^x$ and $\bar X^y .$ 
\item
We start at time $ t= 0 $ with $\bar X^x_0 = x, \bar X^y_0= y.$ 
\item
Take two independent realizations of $Z^x $ and of $Z^y $ and put 
\begin{equation}
\bar X_t^x := Z_t^x , \; \bar X_t^y := Z_t^y \mbox{ for all }  0 \le t < T_1^{[n]} .
\end{equation}
Notice that $ T_1^{[n]} $ is independent of the rhs of \eqref{eq:z} and exponentially distributed with parameter $\Gamma_n  .$ We put 
$$ \bar X^x_{T_1^{[n]} -  } := Z^x _{T_1^{[n]}  - } , \;  \bar X^y_{T_1^{[n]} -  } := Z^y _{T_1^{[n]}  - } .$$ 
On $\bar X^x_{T_1^{[n]} -  } = x'  $ and $\bar X^y_{T_1^{[n]} -  } = y' $ we do the following. 
\item
We choose a uniform random variable $ U_1 , $ uniformly distributed on $ [0, 1 ], $ independently of anything else. 
\item
On $ U_1 = u , $ we choose a random variable $ V_1 \sim Q ( (x',y', u) , \cdot  ) $ and we put 
\begin{equation}\label{eq:vk}
 (\bar X^x_{T_1^{[n]} } , \bar X^y _{T_1^{[n]} }) := V_1 .
\end{equation}
We then restart the above procedure at item (2) with the new starting point $V_1 $ instead of $(x,y) .$ 
\end{enumerate}

Write $( \bf X^x_t , X^y _t)  $ for the $2 d +1-$dimensional process with additional color $ U_k , $ defined by 
$${ ( \bf X^x_t , X^y _t)  }= \sum_{ k \geq 0 } 1_{ [T_k^{[n]} , T_{k+1}^{[n]} [} (t) ( \bar X^x_t,\bar X^y_t,  U_k ) ,$$
keeping trace of the additional color $U_k.$ In the above formula, we put $ U_0 := 1 $ (during the interval $ [0, T_1^{[n ]}[, $ no coupling attempt is made).  Let
\begin{equation}\label{eq:tauc}
 \tau_c :=  \inf\{ T_k^{[n]} , k \geq 1 : (\bar X^x_{T_k^{[n]} - }, \bar X^y_{T_k^{[n]} - }) \in C \times  C , U_k \le  \beta \} ,
\end{equation}
which is the coupling time of the process. It is clear that by the structure of the splitting kernel $Q ( (x',y', u) , \cdot  ),$ if $ \tau_c < \infty , $ then 
$$ \bar X^x_{\tau_c  } =  \bar X^y_{\tau_c    }  \sim \nu  .$$
Once the two trajectories have met at time $ \tau_c,$ by the Markov property, we may merge them into a single one, and there is no need to continue the above construction, that is, we apply the construction (1) -- (5) described above only up to the time $ \tau_c.$ 

It is clear that in this way the speed of convergence to equilibrium of the process is determined by the moments of the coupling time $\tau_c.$ 
In particular, in order to prove that $ \tau_c < \infty $ almost surely, we have to ensure that joint visits of the set $C  $ by $\bar X_t^x $ and $ \bar X_t^y $ do indeed happen. This is granted by a Lyapunov condition plus a control argument that will be developed in the next section. These arguments will not only imply the finiteness of the coupling time, but also a control of its moments.

\subsection{Lyapunov function }
We introduce the operator 
$$L f(x) = \sum_{i=1}^d \frac{\partial f}{\partial x_i } (x)g^i  ( x) + \frac{1}{2} \sum_{i, j = 1}^d \frac{\partial^2 f}{\partial x_i \partial x_j  } (x) a^{i, j }  ( x)
+ \int_{\R^d } [ f ( x + c (  z, x)) - f(x)] \gamma (  z, x) \mu (dz ),
$$
where   for $ 1 \le i, j \le d, $ 
$  a^{i, j } (x)  =  \sum_{l=1}^m\sigma_l^i \sigma_l^j (x)   ,$ for sufficiently regular test functions $f.$ We impose
\begin{ass}\label{ass:drift}
There exists a continuous function $V : \R^d \to [1, \infty [ $ which belongs to the domain $ {\mathcal D} ({L}) $ of the extended generator ${L}$ of the process $\bar X,$ and constants $ b, c > 0  $ such that for any $x \in \R^d ,$ 
\begin{equation}\label{eq:driftcond0}
{L} V (x) \le - b  V(x) + c \mathds{1}_K  (x) ,
\end{equation}
where $K \subset  \R^d $ is a compact.
\end{ass}

\begin{ex}
Suppose that there exists a compact $ K \subset \R^d, $ such that 
\begin{equation}
Tr ( a (x)) + 2 < g ( x) , x> + 2 \int_{\R^d } < x + c(z,x), c (z, x) > \gamma (z,  x ) \mu (dz )
 \le - c |x| ^2 , 
\end{equation}
for all $x \in K^c .$ Then \eqref{eq:driftcond0} holds for $ V(x) = |x|^2 . $ 

We refer to Section 4 of \cite{Victor} for a detailed discussion of other conditions implying \eqref{eq:driftcond0}. 
\end{ex}

\begin{rem}
Theorem 4.1 of Douc et al. \cite{DFG09} (2009), applied to $ \Phi ( x) = b x $ and $ \delta = 0, $  shows that
\eqref{eq:driftcond0} implies
$$ E_x (e^{ b \tau_K} ) \le V(x) , \mbox{ for } \tau_K = \inf \{ t \geq 0 : \bar X_t \in K\} . $$ 
\end{rem}

\begin{cor}
Under Assumption \ref{ass:drift}, for any coupling of $\bar X_t^x $ and $ \bar X_t^y $ and for   $ \tau_{K \times K} := \inf \{ t \geq 0 : (\bar X_t^x , \bar X_t^y ) \in K \times K \}, $ we have 
$$ E_{x,y} ( e^{b \tau_{K \times K}})  \le V(x) + V(y) + C.$$
\end{cor}

\begin{proof}
If suffices to define the $2d-$dimensional Lyapunov function $ \bar V (x, y ) := V(x) + V(y) $ and to check that \eqref{eq:driftcond0} holds for $ \bar L $ where $ \bar L$ denotes the generator of the process $ (\bar X_t^x, \bar X_t^y ).$ 
\end{proof}

As a consequence, under Assumption \ref{ass:drift}, two copies of the process visit the compact $K$ {\bf at the same time} infinitely often, almost surely. 

\subsection{Control}
Once the two copies of the process have entered the compact $K, $ we have to steer them to the target set $C $ appearing in the Doeblin lower bound \eqref{eq:minoration}. This is related to the control properties of the process $ \bar X.$ Since the process $\bar X$ is of infinite jump activity, we start by approximating it by a finite activity process in the following way. 

For any subset $ G \subset \R^d  $ with $\mu ( G) < \infty ,$ we define the process $\XX $ by 
\begin{multline}
\XX_t = x + \sum_{l=1}^m \int_0^t \sigma_l ( \XX_s ) d W_s^l + \int_0^t g ( \XX_s) ds \\
+ \int_{[0, t ]} \int_G \int_{\R_+} c(z, \XX_{s-}) 1_{ \{ u \le \gamma (z, \XX_{s-})\}} N (ds, dz, du ) .
\end{multline}

Then we know that 

\begin{prop}\label{prop:control}[Lemma 6 of \cite{ballyetc}]
Grant Assumption \ref{conditionsbis}. There exists a constant $C> 0 $ such that for any $x \in \R^d $ and $ T > 0, $ 
\begin{equation}\label{eq:control1}
P_x ( \sup_{t \le T} | \XX_t - \bar X_t | \geq \varrho ) \le \frac{T  e}{\varrho} \exp \left( C  T \left[ \sum_l \| \nabla \sigma_l \|_\infty + \| \nabla  g  \|_\infty  + C_{\mu } (\gamma, c )\right]^2    \right)   \alpha ( G^c )  ,
\end{equation}
where $ C_{\mu } (\gamma, c )   $ is defined in \eqref{eq:cmu} and where $ \alpha(G^c) : =\sup_{ x \in \R^d } \int_{G^c} | c (z, x) | \gamma ( z, x) \mu (dz) < \infty 
$ by Assumption \ref{conditionsbis}.
\end{prop}

\begin{cor}\label{cor:412}
Under the conditions of Proposition \ref{prop:control}, for any fixed time horizon $ T$ and any $ \varrho > 0 $  there exists $ G_T \subset \R^d  $ such that for all $ G $ with $G_T \subset G,$ 
\begin{equation}\label{eq:control2}
 \inf_{ x \in \R^d } P_x (  \sup_{t \le T} | \XX_t - \bar X_t | \le  \varrho ) > 0 .
\end{equation} 
\end{cor}

In the following, we shall choose $ \varrho := \eta/ 4$ (recall \eqref{eq:minoration}) and $T := 1 .$ Fix $ G$ such that \eqref{eq:control2} holds with these parameters. 

In a next step we will give conditions ensuring that $ \inf_{x \in K } P_x ( |\XX_1 - x_0 | \le \eta/4 ) > 0 .$ To do so, let us introduce the following objects. Write $\,\tt H\,$ for the Cameron-Martin space of measurable functions ${h}:[0,1]\to \R^k $ having absolutely continuous components ${ h}^\ell(t) = \int_0^t \dot h^\ell(s) ds$ with $\int_0^{1}[{\dot h}^\ell]^2(s) ds < \infty$, $1\le \ell\le k$. For $x\in \R^d $ and ${ h}\in{\tt H}$, consider the deterministic system $ \varphi^{ (h, s,x ) } $ solution of 

\begin{equation}\label{eq:generalcontrolsystem}
 \varphi^{ (h, s,x ) } (t)   =  x + \int_s^t g  (  \varphi^{ (h, s,x ) } (u)  ) du +  \sum_{\ell=1}^k  \int_s^t \sigma_\ell ( \varphi^{ (h, s,x ) }(u) ) \dot h^\ell (u) du. 
\end{equation}
If $ s = 0, $ we write for short $ \varphi^{(h,x)} $ instead of $ \varphi^{(h,0, x) }.$ 

We will impose either the following assumption of strong controllability 

\begin{ass}\label{ass:strongcontrol}
For all $ x \in K, $ there exists $ h \in \tt H $ such that 
$$ \varphi^{(h, x)} (1 ) = x_0 .$$
\end{ass}

Assumption \ref{ass:strongcontrol} is satisfied e.g.\ if the matrix $a$ defined through $ \sum_{l=1}^k \sigma_l^i \sigma_l^j  =   a^{i j }  $ for all $ 1 \le i , j \le d,$ is positive definite on $ K  $ (here we suppose w.l.o.g. that $ x_0 \in  int K $).  %Weaker conditions implying the above assumptions are possible; indeed, in the case of {\it degenerate diffusions,} the above assumption is implied by the strong Hoermander condition.  (WHOM TO CITE ?) \\
However, if Assumption \ref{ass:strongcontrol} does not happen to be satisfied, we may introduce a weaker condition taking into account the jumps of the process $ \XX.$ For that sake, fix some $n \geq 1 , $ a sequence $ 0 < t_1 < \ldots < t_n < 1 $ as well as a sequence $  (z_1, \ldots , z_n ) $ of elements $ z_k \in G .$ Write for short $ \zz = (z_1, \ldots , z_n ) , \ttt = (t_1, \ldots, t_n ) .$ Consider finally a sequence $ \hh := (h_1, \ldots , h_n )  $ of elements of ${\tt H}$ and introduce the skeleton process $ x_t = x_t ( x, \ttt, \zz, \hh ) $ which is defined on $ [0, 1 ]$ as follows. 
$$  x_t = x_t ( x, \ttt, \zz, \hh )  = \left\{ 
\begin{array}{ll}
\varphi^{(h_1, x)} ( t) &0 \le  t < t_1 \\
x_{t_k-}  + c( z_k, x_{t_k-}   )  & t = t_k , 1 \le k \le n , \\
\varphi^{(h_k, x_{t_k} )} (t- t_k ) & t_k \le  t < t_{k+1 }\wedge 1 , 1 \le k \le n 
\end{array}
\right\} ,$$
where we put $t_{n+1} := \infty $ for convenience. Finally we put $x_1 = x_1 ( x, \ttt, \zz, \hh ) = \varphi^{(h_n, x_{t_n} )} (1- t_n ).$

Then we suppose 
\begin{ass}\label{ass:weakcontrol}
The process $ \XX$ has a minimal jumping rate, i.e., 
$$ \gamma (z, x)  > 0 \mbox{ for all $x$ and for all $ z \in G$},$$ and
for all $ x \in K, $ there exist $n \in \N$ and sequences $ \ttt, \zz , \hh $ such that  $ z_1, \ldots, z_n \in supp ( \mu) \cap G  $ and 
$$ |x_1 ( x, \ttt, \zz, \hh ) - x_0 | \le \eta /4 .$$
\end{ass}

\begin{theo}\label{theo:controlXbar}
Suppose that Assumption \ref{conditionsbis} holds. Grant either Assumption \ref{ass:strongcontrol} or \ref{ass:weakcontrol}. Then 
$$ \inf_{ x \in K } P_x ( | \XX_1 - x_0 |\le \eta/4 ) > 0 .$$ 
\end{theo}

\begin{proof}
Recall that $ \sup_{z, x } \gamma ( z, x) \le \Gamma < \infty$ by Item (4) of Assumption \ref{conditionsbis}. 
Therefore, the jumps of $\XX $ occur at most at the jump times of a rate $\mu ( G) \Gamma -$Poisson process that we shall call $ J.$ We work conditionally on $ J_1 = n $ and on the choice of jump times $ T_1 = t_1 < T_2 = t_2 < \ldots < T_n =t_n < 1.$ 

On $\{ J_1 = 0 \}, $  $ \XX $ does not jump during $ [0, 1 ] $ and thus 
\begin{equation}\label{eq:nojump}
 \XX_t = Y_t , \mbox{ for all } t \le 1, 
\end{equation} 
where 
\begin{equation}\label{eq:SDEsanssauts}
  Y_t = x + \sum_l \int_0^t  \sigma_l (Y_s) d W^l_s + \int_0^t g (Y_s) ds =: \Phi_t ( x) .
\end{equation}
Here, $ \Phi_t( x) $ denotes the stochastic flow associated to the above stochastic differential equation. Notice that under our assumptions, $ x \mapsto \Phi_t( x) $ is continuous, see e.g.\ \cite{kunita}. 
Therefore, under Assumption \ref{ass:strongcontrol}, we may conclude as follows: 
$$ P_x ( | \XX_1 - x_0 |\le \eta/4 ) \geq  P_x ( | Y_1  - x_0 |\le \eta/4 ; J_1 = 0  )  = P_x ( | Y_1  - x_0 |\le \eta/4 ) \cdot P ( J_1 = 0 ) > 0 ,$$
due to the support theorem for diffusions, which implies the assertion since $ x \mapsto  \Phi_1 (x) = Y_1  $ is continuous and since $ K$ is compact. 

Suppose now that Assumption \ref{ass:weakcontrol} holds. We work conditionally on $ J_1 = n $ and on $ T_1 = t_1 < \ldots <   T_n = t_n < 1 $ such that $ n , \ttt $ satisfy Assumption \ref{ass:weakcontrol}. Our goal is to construct a version of $ \XX_1, $ conditionally on these choices, which is continuous in the starting point $x.$ This construction relies on the so-called ``real-shocks''-representation of $ \XX_t $ which we define now (cf.\ also to Section 2.2.3 in \cite{ballyetc}).  

During this construction,  we will choose successively random variables $ Z_1, \ldots, Z_n $ and define a process $ x_t ( x, Z_1^n ) ,$ depending on these choices, where $Z_1^n = (Z_1, \ldots , Z_n ) ,$ for $ 0 \le t \le 1.$ This process is defined recursively as follows. Firstly, we put 
$$ x_t =  \Phi_t ( x ) , \mbox{ for all } 0 \le t < t_1 .$$
Then, conditionally on $x_{t_1 -}  = y_1, $ we choose a random variable $ Z_1 $ with law $q_G ( z, y_1 ) \mu^* ( dz) , $ where  
for some fixed $ z^* \in G^c ,$ $ \mu^* (dz ) = \mu ( dz ) + \delta_{z^* } (dz ) $ and 
$$ q_G ( z,y) = \Theta_G ( y) 1_{ z^*} ( z) + \frac{1}{\mu (G) \Gamma } 1_G ( z) \gamma ( z, y) ,$$
 with 
$$ \Theta_G ( y) = 1 - \frac{1}{\mu ( G) \Gamma } \int_G \gamma ( z, y ) \mu (dz) .$$
Then we put 
$$x_{t_1 } := x_{t_1 -} + c(  Z_1, x_{t_1 -} ) 1_G ( Z_1), x_t =  \Phi_{t  - t_1} ( x_{t_1} ) , \mbox{ for all } t_1 \le t < t_2 , $$
and we proceed iteratively by choosing, conditionally on $x_{t_2 -}  = y_2, $ a random variable 
$$ Z_2 \sim q_G ( z, y_2) \mu^* (dz) ,$$
and so on. Finally, we obtain a terminal value $x_1  = \Phi_{1 - t_n } ( x_{t_n} ) .$ It is easy to check that 
$$ {\mathcal L} ( x_1 ( x, Z_1^n  )) = {\mathcal L} ( \XX_1 | J_1 = n , T_1 = t_1, \ldots, T_n = t_n) .$$
Due to the support theorem for diffusions and by continuity of $c (x, z),$ we clearly have that $x_1 ( x,  \ttt, \zz, \hh ) \in supp ({\mathcal L} ( x_1 ( x, Z_1^n  ))) .$ Therefore, Assumption \ref{ass:weakcontrol} implies that for all $x \in K, $
$$ P ( | x_1 ( x, Z_1^n  ) - x_0| \le \eta/4 ) > 0 .$$
The important point is now that the above construction ensures the continuity of 
$$ x \mapsto x_1 ( x,  Z_1^n  ).$$
Thus, by continuity in $x$ and compactness of $ K,$ 
$$ \inf_{x \in K }  P ( | x_1 ( x, Z_1^n  ) - x_0| \le \eta/4 ) > 0 $$
implying the assertion.
\end{proof}

We may now conclude with our main result of this section. Introduce 
$$ \tau_1^{[n]}   = \inf\{ T_k^{[n]} , k \geq 1 : (\bar X^x_{T_k^{[n]} - }, \bar X^y_{T_k^{[n]} - }) \in C \times  C \} .$$

\begin{prop}\label{prop:goodcontrol}
Grant Assumptions \ref{conditionsbis}, \ref{conditions2bis}, \ref{ass:drift} and \ref{ass:strongcontrol} or \ref{ass:weakcontrol}. Then there exists $n_0 $ such that for all $ n \geq n_0, $ for all $ p \geq 1, $ there exists a constant $ C = C( p) $ with 
\begin{equation}\label{eq:goodcontrol}
 E_{x,y} ( (\tau_1^{[n]} )^p ) \le C(p )  [V (x) + V(y ) ] .
\end{equation}   
\end{prop}

The proof of Proposition \ref{prop:goodcontrol} is given in the Appendix. 

%We are now able to state our main result on the speed of convergence to equilibrium of $ \bar X_t.$ This is done in the next subsection.

\subsection{Speed of convergence to equilibrium}
Let us summarize all assumptions needed so far.

\begin{ass}\label{ass:Final!}
\begin{enumerate}
\item We impose Assumption \ref{conditionsbis}, implying the existence of a unique strong solution of \eqref{eq:sde0}.
\item
We impose the non-degeneracy condition Assumption \ref{conditions2}. 
\item
We impose Assumption \ref{conditions2bis} implying that the jump measure is sigma-finite. 
\item
We impose the local Doeblin condition \eqref{eq:minoration}. 
\item 
We impose the Lyapunov type condition Assumption \ref{ass:drift}. 
\item 
We impose the controllability condition Assumption \ref{ass:strongcontrol} or \ref{ass:weakcontrol}. 
\end{enumerate}
\end{ass}

Let $ {\mathcal G} := \{ f : \R^d \to \R  \mbox{ measurable} : \|f\|_\infty \le 1 \}  $ and introduce for any two probability measures $\mu $ and $\nu $ on $(\R^d, {\mathcal B}(\R^d ) )$ the distance 
$$ d_{\mathcal G} ( \mu, \nu ) := \sup_{ f \in {\mathcal G} } | \mu  (f) - \nu (f) | ,$$
which is the total variation distance between $ \mu $ and $ \nu .$ 
 
Write $P_t$ for the transition semigroup of the limit process, i.e., $P_t f (x) = E (f (\bar X_t^x )) .$ Then we have the following result.

\begin{theo}\label{theo:tv}
Grant Assumption \ref{ass:Final!}. Then the process $ \bar X_t$ is positively Harris recurrent with unique invariant probability measure $\pi.$ Moreover, for all $ p \geq 1, $ there exists a constant $C(p) $ such that 
$$ | P_t f(x) - P_t f (y )|  \le [ V(x) + V(y ) ] \| f\|_\infty \frac{C(p)}{t^p } .$$ 
Finally, if $(X_t)_{t \geq 0} $ is any stochastic process defined on $ (\Omega , {\mathcal A}, P)$ satisfying $ \sup_{s \geq 0} E ( V ( X_s ) ) < \infty  $ and if we denote $ \mu_s = {\mathcal L} ( X_s ) $ its law at time $s,$ then for all $ p \geq 1, $ 
$$
 d_{\mathcal G} ( \mu_s P_t  , \pi  ) \le C (p )  t^{ - p },
$$
for all $ s , t \geq 0 ,$ and for a suitable constant $ C (p) $ depending on $p.$  
\end{theo}

\begin{proof}
The Harris recurrence of $ \bar X_t$ follows from Theorem 2.12 of \cite{Victor}. To prove the second assertion, the main point of the proof is the fact that Proposition \ref{prop:goodcontrol} implies that 
$$ E_{x, y }( \tau_c^p)  \le C(p ) [ V (x) + V(y) ] .$$
Indeed, this is a direct consequence of  \eqref{eq:goodcontrol} together with the definition of the coupling time $\tau_c$ in \eqref{eq:tauc}.
Then 
$$  | P_t f(x) - P_t f (y )|  \le 2 \|f\|_\infty P_{x, y } (\tau_c > t) \le 2 \|f\|_\infty \frac{ E_{x, y } (\tau_c^p) }{t^p } $$
allows to obtain the second assertion. Moreover, notice that by Theorem 4.3 of \cite{Me93}, Assumption \ref{ass:drift} implies in particular that, once we have proven the unique ergodicity of $ \bar X_t$ with invariant probability measure $\pi, $ we necessarily have that $ \pi ( V) < \infty .$ Therefore we obtain, integrating the first assertion against $ \mu_s ( dx) $ and $ \pi ( dy ) , $ that 
$$ d_{\mathcal G} ( \mu_s P_t  , \pi  ) \le [ \pi ( V) + E ( V ( X_s ) ) ] \frac{C(p)}{t^p } ,$$
implying the assertion, since by assumption,  
$ \sup_s E ( V ( X_s ) ) < \infty . $ 
\end{proof}

We are now ready to study the longtime behavior of a time inhomogeneous Markov process having jumps with position dependent jump rate and infinite activity jump activity. 

\section{Longtime behavior of time inhomogeneous PDMP's}\label{sec:2}
We now turn to the main goal of this paper, the study of the longtime behavior of solutions of \eqref{eq:sde}. In order to grant existence and uniqueness of the solution of \eqref{eq:sde}, we impose the following conditions on the coefficients $b,    c  $ and $ \gamma  .$  

\begin{ass}\label{conditions}
\begin{enumerate}
\item
$b(t, x ) $ is globally Lipschitz continuous in $x, $ uniformly in time, that is,  there exists a constant $L_b$ such that $ \sup_{t \geq 0 } | b ( t, x) - b (t, y ) | \le L_b | x - y | ,$  for all $ x, y \in \R^d .$ 
\item
$c$ and $ \gamma$ are Lipschitz continuous with respect to $x,$ uniformly in time, i.e.\ for all $ t > 0, $
$$ |c (t, z, x ) - c (t, z, x') | \le L_c( z ) | x - x ' | \mbox{ and } |\gamma (t,  z, x ) - \gamma (t, z, x') | \le L_{\gamma}( z ) | x - x ' |,$$
where $ L_c, L_{\gamma} $ are measurable functions from $  {\R^m}  \to \R_+ .$ 
\item
For all $ T > 0, $ 
$ \sup_{x \in \R^d } \sup_{0 \le t \le T  } \int_{{\R^m} }( L_{c} ( z)  \gamma (t, z,x) +  L_{\gamma} ( z) |c (t, z,x) |) \mu  (d z)  < \infty .$
\item
For all $ T > 0 ,$  we have that $\sup_{ 0 \le t \le T  } \sup_x \int_{{\R^m} }  \gamma (t, z,x)   |c (t, z,x) |  \mu   (d  z )< \infty .$
\end{enumerate}
\end{ass}

Under these assumptions, we may still apply Theorem 1.2.\ of \cite{carl} to guarantee that  \eqref{eq:sde} admits a unique strong non-explosive adapted solution which is Markov, having c\`adl\`ag trajectories.  In the following, for any $ t_0 \geq 0, x \in \R^d , $ we shall write $ X_{t_0, t }^x , t \geq t_0, $ for a version of the above process starting from the position $ x  $ at time $t_0.$

Our aim is to show  that, as $t \to \infty , $ under suitable conditions, the time-inhomogeneous process $X_{t_0, t}^x$ of \eqref{eq:sde} converges to the time homogeneous limit process solving equation \eqref{eq:sde0} of Section \ref{sec:1}. 
In order to identify the limit process, we have to distinguish the three possible jump regimes that we have discussed in the introduction, the slow, intermediate and fast regime. 

\begin{ass}\label{ass:1}
For all $t \geq 0, $ there exists a measurable partition $ (E^l_t, l=1, 2, 3) $ of $ \R^m $ such that $ E_t^i \cap E_t^j = \emptyset $ for all $ i \neq j $ and $ E_t^1 \cup E_t^2 \cup E_t^3 = \R^m,$ with the following  properties.

1. (Fast regime) For all $ x \in \R^d ,$
\begin{equation}\label{eq:et1}
\int_{E_t^1 } c(t,z, x) \gamma (t, z, x) \mu ( dz) = 0 ,  \; \lim_{t \to \infty } \int_{E_t^1 } |c(t,z, x)|^3  \gamma (t, z, x) \mu ( dz) = 0 . 
\end{equation}
Moreover, there exists a measurable function $a : \R^d \to \R^{d \times d } $ such that 
\begin{equation}\label{eq:at}
 a^{i j } (t, x) = \int_{E_t^1 } c^i (t, z, x ) c^j (t, z, x) \gamma ( t, z, x ) \mu (dz ), 1 \le i, j \le d, 
\end{equation}
satisfies
$\sup_{ t \geq t_0}  | a (t, x ) - a (x)  | \to 0 $
as $t_0 \to \infty .$ \\

2. (Intermediate regime) For all $x \in \R^d , $ 
\begin{equation}\label{eq:et2}
\lim_{t \to \infty } \int_{E_t^2 } |c(t,z, x)|^2  \gamma (t, z, x) \mu ( dz) = 0 . 
\end{equation}
Moreover, there exists a measurable function $b_2 : \R^d \to \R^{d  } $ such that 
\begin{equation}\label{eq:bt}
 \tilde b (t, x) = \int_{E_t^2 } c (t, z, x )  \gamma ( t, z, x ) \mu (dz ) 
\end{equation}
satisfies
$\sup_{ t \geq t_0}  | \tilde b (t, x ) - b_2 (x)  | \to 0 $
as $t_0 \to \infty .$ \\

3. (Slow regime) There exist measurable functions $\gamma (z, x ) \geq 0 $ and $ c(z, x) $ such that 
\begin{equation}\label{eq:ct}
\sup_{ t \geq t_0}  \int_{E_t^3}  \Big( | \gamma (t, z, x ) -  \gamma ( z, x) | 
+  \gamma ( z, x )   | c ( t, z, x ) -  c ( z, x) | \Big)   \mu (dz ) \to 0 
\end{equation}
as $t_0 \to \infty , $ for all $ x \in \R^d .$ \\

4. There exists a measurable function $b_1 : \R^d \to \R^{d  } $ such that 
$
 \sup_{ t \geq t_0}  | b (t, x ) - b_1 (x)  | \to 0 
$
as $t_0 \to \infty .$ \\

5. Introducing 
\begin{multline}
 \varepsilon ( x, t_0) = \sup_{ t \geq t_0 }  \int_{E_t^1}  | c (t, z, x) |^3 \gamma (t, z, x) \mu (dz) + \sup_{t \geq t_0}  \int_{E_t^2}  | c (t, z, x ) |^2 \gamma ( t, z, x ) \mu  (dz ) \\
 + \sup_{ t \geq t_0} [ | a (t, x ) - a (x)  |  +|b ( t, x ) + \tilde b ( t, x)  -  g(x)|  ]  \\
+ \sup_{ t \geq t_0}  \int_{E_t^3}  \Big( | \gamma (t, z, x ) -  \gamma ( z, x) | 
+  \gamma ( z, x )   | c ( t, z, x ) -  c ( z, x) | \Big)   \mu (dz ) ,
\end{multline}
there exist $C,  r > 0 $ such that $ \varepsilon (x,t) \le C [1+|x|] e^{-rt}  .$\\

6. Let $\sigma_l , 1 \le l \le k, $ be such that $  a^{i j } (x)  =  \sum_{l=1}^m\sigma_l^i \sigma_l^j (x)  $ for all $ 1 \le i, j \le d .$ Put $ g ( x) = b_1 ( x) + b_2 (x) .$ Then $\sigma_l, g , c $ and $\gamma $ are such that  Assumption \ref{ass:Final!} holds. 
\end{ass}

With $L_t $ the generator of \eqref{eq:sde} as in \eqref{eq:gen}, it is immediate to see that the following result holds true.

\begin{prop}
Suppose that the coefficients of  the stochastic differential equations \eqref{eq:sde} satisfy Assumption \ref{conditions}. Grant moreover Assumption \ref{ass:1}. Then there exists a constant $C > 0 $ such that for any function $f \in C_b^3 ( \R^d ) $ 
$$ | Lf (x) - L_t f (x)  | \le C e^{-rt } [1+|x|]  \|f\|_{3, \infty } .$$  
\end{prop}

Therefore, the infinitesimal generator of \eqref{eq:sde} converges to the one of the limit process, if Assumption \ref{ass:1} holds. If the limit semigroup $P_t f(x) = E ( f ( \bar X_t^x )) $ satisfies suitable regularity conditions, this implies the convergence of the associated semigroups (see e.g.\ \cite{ballyetc}, \cite{bouguetetal}). This regularity of $P_t f (x)$ is actually delicate to show due to the presence of the position dependent jump rate $\gamma ( z, x ) .$ We refer to  \cite{ballyetc} for a thorough study on which we rely in the sequel. 

\subsection{Asymptotic pseudotrajectories}
To prove the convergence of the time dependent process to the limit process, we shall need both assumptions on the coefficients of the limit semigroup as well as on the time dependent approximating semigroup $L_t .$ 

{\bf Conditions on the limit process.} To state the conditions on the limit process in the presence of the position dependent jump rate $ \gamma ( z, x) ,$ we have to introduce the following notation. 
For a function $ f : \R^m  \times \R^d \to   \R$ which is $ q -$times differentiable with respect to $x, $ we write for any $ p \geq 1,$ for any time horizon $ T > 0$  and for any constant $C > 0, $ 
\begin{eqnarray}
|f|_{p} &=& \sup_{x \in \R^d } \left( \int_{\R^m }  | f (z, x) |^p \gamma ( z, x) \mu ( dz)\right)^{1/p},  \\
{[ f ]}_{ p } &=& \sup_{1 \le p' \le p } |f|_{ p'} , \\
\theta_{(q,p)}  &=& 1 + \| \sigma \|_{2, q, \infty} + \| g \|_{2, q , \infty}  + \sum_{ 2 \le |\alpha | \le q } [ \partial^\alpha_x c]_{ p},\\
\alpha_p  &=& \| \nabla \sigma\|^2_\infty + \| \nabla g \|_\infty + [ \nabla_x c]^p_{ p } ,\label{eq:alphap}\\
\alpha_{q, p } ( C, T  ) &=& C \theta^q_{q, p }  \exp{ (C T  q \sum_{ 1 \le n \le q } \frac1n \alpha_{ p \cdot q }  ) } ,\\
\Gamma_{ q} ( \gamma) &=& \sup_{x \in \R^d } \sum_{l=1}^q \sum_{ 1 \le |\alpha | \le l } \left( \int_{\R^m} | \partial^\alpha_x \ln \gamma (z, x) |^{l/ | \alpha|} \gamma (z, x) \mu (dz) \right)^{q/l} .
\end{eqnarray}
In the following, we will apply the above notations with $ q=3, p=12$ and write
\begin{equation}\label{eq:q3}
 Q_3(P, T ) := \alpha^{6 }_{3, 12 } ( C, T  ) \times \left( 1+ \Gamma_{ 3} ( \gamma) + \sum_{ 0 \le |\beta | \le 3 } [ \partial^\beta_x \ln \gamma ]_{ 12 } \right)^3 .
\end{equation}

{\bf Conditions on the time inhomogeneous approximating process.} To begin with, we impose the following condition implying that the time inhomogeneous process  $ X_{t_0, t }^x $ is $1-$ultimately bounded. 

\begin{ass}\label{unifbounded}
There exists a constant $C > 0 $ such that for all $ t_0 > 0 , $ 
\begin{equation}\label{eq:ultimativelybounded}
 \sup_{t \geq t_0 } E ( | X_{t_0, t }^x  |) \le C  ( 1 + |x|)  .
\end{equation} 
\end{ass} 

\begin{rem}
To check \eqref{eq:ultimativelybounded}, it is sufficient to impose a Lyapunov condition. 
Suppose e.g.\ that there exists a function $ V : \R^d \to \R_+ $ with $ E ( V ( X_{t_0, t }^x )) < \infty $ for all $ t \geq t_0 $ and with $ V( x) \geq |x| $ if $ |x| \geq K ,$ where $ K$ is some fixed constant. Assume that $V$ is a Lyapunov function in the sense that there exist $ \alpha , \beta > 0 $ such that 
\begin{equation}\label{eq:lyapunovlt}
 L_t V ( x) \le - \alpha V( x) + \beta 
\end{equation}  
for all $ t \geq 0 . $ Then \eqref{eq:ultimativelybounded} holds.  
\end{rem}

Finally, in order to control the regularity of the approximating semigroup, we introduce
\begin{eqnarray*}
\Phi_1 (t, z, x)&=& \left[ | \nabla_x \gamma (t, z, x) | |c ( t, z, x) |^2  + (   | \nabla_x c |^2 + |c|^2 ) \gamma (t, z, x)\right] 1_{E_t^1 } (z) ,\\
\Phi_2 (t, z, x)&=& \left[ | \nabla_x \gamma (t, z, x) | |c ( t, z, x) | + ( | \nabla_x c | |c| + | \nabla_x c | ) \gamma (t, z, x)\right]  1_{E_t^2 \cup E_t^3  } (z),
\end{eqnarray*}
and we suppose that 
\begin{equation}\label{eq:ct}
C_{t_0} := \sup_{ t \geq t_0} \sup_{x \in \R^d } \sum_{i=1}^2 \int_{\R^m} \Phi_i ( t, z, x) \mu  ( dz) < \infty 
\end{equation}
for some $t_0 > 0.$ 

{\bf Asymptotic pseudotrajectories}
We state our convergence result in terms of asymptotic pseudotrajectories introduced in Bena\"im and Hirsch  \cite{bh96} (1996), see also Bena\"im et al. \cite{bouguetetal} (2016). This notion provides an efficient tool to describe the long time behavior of time dependent processes. Consider the class of test functions 
$$ {\mathcal F} = \{ f : \R^d \to \R : \| f\|_{3, \infty} \le  1 \} ,$$
and introduce for any two probability measures $ \mu $ and $\nu $ on $(\R^d , {\mathcal B} ( \R^d ) ) $ the associated distance
$$ d_{\mathcal F} ( \mu , \nu ) := \sup_{ f \in {\mathcal F}} | \mu ( f) - \nu ( f) | .$$
For $X_t $ a solution of \eqref{eq:sde}, starting from any arbitrary initial law $ {\mathcal L} ( X_0) , $ let 
$$ \mu_t := {\mathcal L} (X_t) $$
and recall that $P_t$ denotes the transition semigroup of the limit process \eqref{eq:sde0}. The following result is our main result.

\begin{theo}\label{cor:1}
Suppose that the coefficients of \eqref{eq:sde} satisfy Assumption \ref{conditions} and that $ E ( |X_0 |) < \infty .$ Suppose moreover that $ Q_3 (P, T ) < \infty $ for all $ T > 0 $ and that $C_{t_0 } < \infty $ for some $t_0 > 0 .$ Finally, grant Assumptions \ref{ass:1} and  \ref{unifbounded}. Then there exists a constant $M_1$ such that for any $ T < \infty , $ 
\begin{equation}\label{eq:217}
\sup_{s \le T} d_{\mathcal F} ( \mu_{t + s } , \mu_t P_s ) \le C e^{ M_1  T}  \int_t^{t+T} e^{-rs}  ds  .
\end{equation}
In particular,  
$$ \lim_{t \to \infty } \sup_{s \le T} d_{\mathcal F} ( \mu_{t + s } , \mu_t P_s ) = 0,$$
thus $ ( \mu_t)_t $ is an asymptotic pseudotrajectory of $ (P_t)_t $ in the sense of \cite{bh96}. 
\end{theo}

The proof of Theorem \ref{cor:1} is given in the Appendix.

Since according to Theorem \ref{cor:1}, $ X_{t_0, t}^x $ is a good approximation of $ \bar X_{t_0, t }^x $ as $t_0$ tends to infinity, it is natural to study to which extent $ X_{t_0, t}^x  $ approaches the invariant regime of the limit process, as time tends to infinity. Recall that $P_t$ denotes the limit semigroup and $\pi $ the associated invariant probability measure, which exists according to Theorem \ref{theo:tv}. Finally, recall that $ {\mathcal G} = \{ f : \R^d \to \R \mbox{ measurable such that } \| f \|_\infty \le 1 \} .$ The following theorem is an immediate consequence of Theorem \ref{theo:tv}. 

\begin{theo}\label{theo:2}
Grant Assumption \ref{ass:Final!} and the assumptions of Theorem \ref{cor:1}. Let $V(x)$ be the Lyapunov function of Assumption \ref{ass:drift} and suppose that there exists a constant $C$ such that for all $ t_0 > 0 , $ 
$$ \sup_{t \geq t_0 } E ( V( X_{t_0, t }^x ) ) \le C  ( 1 + V(x) )  .$$
Then for all $ p \geq 1 ,$ for a  constant $C = C(p) $ depending on $p$ and on ${\mathcal L} ( X_0) ,$  we have
\begin{equation}\label{eq:controlg}
 d_{\mathcal G} ( \mu_s P_t  , \pi  ) \le C t^{ - p },
\end{equation} 
for all $ s , t \geq 0 . $  
\end{theo}

As a consequence, we can then show that $ \mu_t $ converges, as $t \to \infty, $ to the invariant measure of the limit semigroup, in the following sense. 

\begin{cor}\label{cor:3}
Grant the assumptions of Theorem \ref{cor:1}.  Then for all $ p \geq 1 , $ there exists a constant $C= C(p, M_1 )   > 0, $ such that for all $ t \geq 0, $ 
\begin{equation}
d_{ \mathcal F} ( \mu_t, \pi ) \le C t^{- p} .
\end{equation}
\end{cor}

\begin{proof}
Fix some $ 0 < \alpha < 1 .$ Then, since $ {\mathcal F} \subset {\mathcal G} ,$ 
$$ d_{ \mathcal F} ( \mu_t, \pi ) \le d_{\mathcal F} ( \mu_t , \mu_{ \alpha t } P_{t- \alpha t } ) + d_{\mathcal G}  (  \mu_{ \alpha t } P_{t- \alpha t } , \pi ) .$$
Using \eqref{eq:217} with $T = ( 1- \alpha) t $ and with $ \alpha t  $ instead of $ t $ by   we obtain 
$$d_{\mathcal F} ( \mu_t , \mu_{ \alpha t } P_{t- \alpha t } )  \le C e^{ M_1  ( 1- \alpha) t } e^{- r \alpha t }.$$ 
Choose therefore $ \alpha  $ sufficiently close to $1$ such that 
$$ r \alpha  - M_1  ( 1- \alpha)     >  r/2 .$$
For this choice of $\alpha, $ 
$$ d_{\mathcal F} ( \mu_t , \mu_{ \alpha t } P_{t- \alpha t } )  \le C e^{- \frac{r}{2}    t }$$ 
and by \eqref{eq:controlg}, 
$$ d_{\mathcal G}  (  \mu_{ \alpha t } P_{t- \alpha t } , \pi ) \le C (p)   (1 - \alpha )^{-p} t^{-p}  ,$$ 
from which we deduce the result.  
\end{proof}

%\begin{rem}
%The study of non-stationarity is a challenging problem which is very up to date, in particular in statistics of stochastic processes, and a lot of papers have been devoted to this topic during the last decade, see e.g. \cite{JM} or \cite{Dahlhaus2017}  and the references therein. All these papers deal with the time discrete case, and to the best of our knowledge, the results presented above in Theorem \ref{theo:2} and Corollary \ref{cor:3} are one of the first results that allow to deal  with the longtime behavior of time inhomogeneous processes in the framework of continuous time observations. 
%\end{rem}

\section{Examples}\label{sec:4}
\subsection{Hawkes processes with exponential memory kernels and memory of variable length in a mean-field frame}
Hawkes processes are point process models which are very important from a modeling point of view. They have regained a lot of interest in the recent years, in particular in econometrics, as good models to account for contagion risk and clustering arrival of events. They have also shown to be very useful in neuroscience due to their capacity of reproducing both the typical time dependencies observed in spike trains of neurons as well as the interaction structure of neural nets. Originally introduced by  \cite{Hawkes} and \cite{ho} as a model for the appearances of earthquakes, their key feature is the fact that any point event is able to trigger future events -- for this reason, Hawkes processes are sometimes called ``self-exciting point processes''. 

We start by briefly recalling the definition of a Hawkes process. A Hawkes process $Z$ is a  counting process on the real line $\R .$ Its law is characterized by its stochastic intensity processes $ \lambda_t   $ which is defined through the relation 
$ P ( Z \mbox{ has a jump in } ]t , t + dt ] | {\mathcal F}_t ) = \lambda_t  dt , $
where $ {\mathcal F}_t = \sigma (  Z ( ]u, s ] ) , \, -\infty < u <  s \le t ),$ and where 
\begin{equation}\label{eq:intensity0}
\lambda_ t = f  \left(     \int_{]-\infty  , t [} h ( t-s) d Z_s  \right) .
\end{equation}
Here, $f : \R \to \R_+$ is the {\it jump rate function} and $ h : \R_+ \to \R  $ is the {\it memory kernel.} 

For simplicity, in what follows we suppose that $h$ is an exponential kernel, that is, it is of the form 
\begin{equation}\label{eq:Erlang}
h(t)=ce^{- \alpha t}, t \geq 0,
\end{equation}
where $ c \in \R  .$ If $ c > 0 , $ then the process is self-exciting, a negative value of $c$ implies that the process is self-inhibiting. 

Instead of considering a single Hawkes process $Z, $ systems of interacting Hawkes processes display a much richer behavior. So let us consider, inspired by \cite{SusanneEva} and \cite{aaee},  a system of $N$ Hawkes processes $Z^1, \ldots , Z^N ,$ having intensity $ \lambda_t^1, \ldots , \lambda_t^N $ each. We will suppose that the interactions between these processes are of mean field type (see \eqref{eq:intensitymf} and \eqref{eq:int} below), with exponential memory kernel. In many situations it is reasonable to assume that the jump intensity of some of the processes, say of process $Z^1,$ is only influenced by its history since its last jump time. This is what \cite{GL} call ``memory of variable length'', in reminiscence of the so-called ``Variable length Markov chains'' coined by Rissanen \cite{rissanen1983} (1983). More precisely, if we put 
$$ L_t := \sup \{ s \le t :\Delta  Z^1 (s) = 1 \} , $$
which is the last jump of $ Z^1$ before time $t,$ then we suppose that the jump intensity of $Z^1 $ is of the form 
\begin{equation}\label{eq:intensitymf}
\lambda_t^1 = f_1 \left(  \frac{1}{{N- 1}} \sum_{j=2}^N  \int_{]L_t , t [} e^{-  \alpha (t- s)} d Z^j_s\right) .
\end{equation}
In addition, we suppose that the intensity of any of the remaining processes $Z^2, \ldots, Z^N $ is given by
\begin{equation}\label{eq:int}
 \lambda^2_t =\ldots = \lambda^N_ t = f_2 \left(  \frac{b}{\alpha} - \frac{c}{N-1} \sum_{j=2}^N  \int_{]- \infty , t [} e^{-  \alpha (t- s)} d Z^j_s \right) .
\end{equation}

Here,  $f_1 $ and $f_2$ are non-decreasing, strictly positive, lowerbounded, bounded Lipschitz continuous functions having bounded derivative. Moreover, $ \alpha , b , c > 0 $ are fixed constants.

Introduce now
$$ X_t^{1} := \frac{1}{{N- 1}} \sum_{j=2}^N  \int_{]L_t , t ]} e^{-  \alpha (t- s)} d Z^j_s, \; X_t^2 := \frac{b}{\alpha} -  \frac{c}{N-1} \sum_{j=2}^N  \int_{]- \infty , t ]} e^{-  \alpha (t- s)} d Z^j_s,$$
both processes depend implicitly on $ N, $ the number of interacting components in the system. Therefore, we shall write $ X_t^{[N]} := X_t := (X_t^1 , X_t^2 );$ this is a two-dimensional Markov process with drift coefficient
$$ b (N, x) =   - \alpha x + b \left( 
\begin{array}{ll}
0 \\
1
\end{array}
\right) .$$
In order to recognize the different jump regimes, we actually have to guess the fast, intermediate and the slow regime. It turns out that due to the mean field frame, no diffusive regime will appear in the limit (that is, no fast regime is present in this case). The jumps of process $Z^1 $ induce big jumps of $X_t^1. $  Indeed, each time that $Z^1 $ jumps, the process $X^1 $ is reset to $0, $ which is a consequence of its variable length memory structure. All other jumps will lead to a deterministic drift function in the limit.

Therefore, we may choose $ d= 2, m = 1 $ and $\mu (dz ) = dz $ together with $E_t^1 = \emptyset , $ $E_t^2 = ]0, 1 [ ,$ $E_t^3 = ]1, 2 [ ,$ 
$$ \gamma (N, z, x)  = 1_{ ]0, 1 [} (z) (N-1) f_2 (  x^2) +  1_{ ]1, 2 [} (z) f_1 (  x^1) ,   $$
and jump amplitude functions  
\begin{equation}\label{eq:amplitude}
 c (N, z, x)   =  1_{ ]0, 1 [} (z) \frac{1}{{N-1 }} 
\left( \begin{array}{cc}
1 \\
-c
\end{array}
\right)+  1_{ ]1, 2 [} (z) \left( \begin{array}{cc}
-x^1 \\
0
\end{array}
\right) .
\end{equation}

Instead of considering the frame of time inhomogeneous processes, in the present situation it is reasonable to prove the convergence of $X^{[N]} $ to a limit process, as the number of interacting components, $N, $ tends to infinity. 
Firstly, we realize that the ``jump drift'' given by $\int_{E_t^2} \gamma (N, z, x)  c (N, z,  x)dz $ equals
$$ \int_{E_t^2} \gamma (N, z, x)  c (N, z,  x)dz =f_2 (x^2) \left( \begin{array}{cc}
1 \\
-c
\end{array}
\right)  .$$

Therefore, the associated limit process is given by $ \bar X_t = ( \bar X_t^1 , \bar X_t^2 ) , $ where $ \bar X_t^2 $ follows an autonomous deterministic equation given by 
$$ d \bar X_t^2 = - \alpha \bar X_t^2 - c f_2 ( \bar X_t^2) dt + b dt ,$$
and where 
\begin{equation}\label{eq:sdehawkeslimit}
d \bar X^1_t = - \alpha \bar X_t^1 dt +f_2 ( \bar X_t^2) dt    -  \int_{\R_+} \bar X_{t-}^1 1_{ \{ u \le f_1 ( \bar X_{t-}^1 ) \}} \bar N (dt, du ) ,
\end{equation}
with $\bar N (dt, du ) $ a PRM on $ \R_+ \times \R_+ $ having intensity measure $ dt du .$ 
This limit process is a true PDMP having only ``big jumps'' (those of the first coordinate) and evolving in a deterministic manner in between successive jumps. 

Obviously, this process can only be ergodic if the coefficients $ \alpha , c $ and $b$ are such that the autonomous deterministic dynamical system describing $ \bar X_t^2 $ possesses an equilibrium. Since $f_2$ is non-decreasing, a sufficient condition for this is that 
$$ \alpha > 0 .$$ 
Once the second component is at equilibrium, the first component evolves as a renewal process:  the successive visits to the state $ \bar X_t^1 = 0 ,$ that occur at each jump of the process, induce an explicit regeneration scheme. We refer the reader to \cite{evafou} for the study of a (much more complicated) related situation.

Suppose now  that instead of being reset to $0$ after a big jump, the process $X_t^1$ is reset to some random value, that is, we replace \eqref{eq:amplitude} by  

\begin{equation}\label{eq:amplitudebis}
 c (N, z, x)   =  1_{ ]0, 1 [} (z) \frac{1}{{N-1 }} 
\left( \begin{array}{cc}
1 \\
-c
\end{array}
\right)+  1_{ ]1, 2 [} (z) \left( \begin{array}{cc}
-x^1 + \varepsilon (z - 1)\\
0
\end{array}
\right) ,
\end{equation}
for some small $\varepsilon > 0.$ This does not change the evolution of the second coordinate, but for the first one we obtain now 
\begin{equation}\label{eq:sdehawkeslimitbis}
d \bar X^1_t = - \alpha \bar X_t^1 dt +f_2 ( \bar X_t^2) dt    -  \int_{\R_+} [ \bar X_{t-}^1  - \varepsilon z] 1_{ \{ u \le f_1 ( \bar X_{t-}^1 ) \}} \bar N (dt, dz, du ) ,
\end{equation}
with $\bar N (dt, dz, du ) $ a PRM on $ \R_+ \times [0, 1 ] \times \R_+ $ having intensity measure $ dt dz du .$ 

With slight modifications of the tools developed in Section \ref{sec:1}, $ \bar X_t^1 $ can then easily shown to be exponentially ergodic, and it is straightforward to deduce that $ {\mathcal L} ( X_t^{[N]} ) $ is an asymptotic pseudotrajectory of the limit semigroup $ (P_t)_t$ if we choose a joint convergence of $ (N, t ) $ to infinity such that $ N = e^{rt} $ for some $r > 0.  $ In other words, we need to simulate an exponential (in time) number of particles $N$ in order to be sure that the above approximation procedure works -- which is the typical order of magnitude for the speed of convergence in mean field limits.

\subsection{A limit Cox-Ingersoll-Ross type jump process.}

We continue the example given in the introduction. Thus $ d=m=  1 $ and $ \mu ( dz ) = dz . $ For $t > 0$ and for some $r >0,$ 
$$ c (t, z, x ) = 
\frac{\sigma }{2 } e^{-rt}  [ 1_{ ]-3 e^{2rt}  , -2e^{2rt} [ } (z) - 1_{ ] -2e^{2rt} , -e^{2rt}  [ } (z) ]  - a x e^{-2rt}   1_{ ]- e^{2rt}, 0 [ } (z) + \frac{d}{(1 + z)^2 }  1_{ ]0, e^{2rt} [ } (z),$$
$b(t, x ) = b $  and 
$$ \gamma(t, z, x) =  f(x) 1_{ ]-3 e^{2rt}  , - e^{2rt} [ } (z) + 1_{ ]- e^{2rt}, 0 [ } (z)  +   f(x) 1_{ ]0, e^{2rt} [ } (z) , $$
for some constants $ \sigma , d \in \R $ and $ a, b  > 0 ,$ where, $f : \R \to \R_+$ is bounded, having bounded derivative such that $ \inf_{ x \in \R } f( x) > 0 .$  

By construction, jumps coming from ``noise'' $ z \in ]-3  e^{2rt},- e^{2rt} [ $ are centered, that is, 
$$ \int_{-3 e^{2rt}}^{-e^{2rt}} c(t, z, x) \gamma ( t, z , x ) \mu (dx ) = 0, $$ 
and the associated variance is given by $ a (t, x) = \sigma^2   f(x) $ for all $ t, x .$ 
Moreover, for all $ t, x, $ 
$$ \int_{-e^{2rt}}^{0} c(t, z, x) \gamma ( t, z , x ) \mu (dx ) =  - a x $$ 
giving rise to a limit drift $-a x .$ Finally, the jumps produced by noise $z > 0 $ survive in the limit process -- this corresponds to the slow jump regime. It is straightforward to show that the associated limit process is a Cox-Ingersoll-Ross type jump process given by  
$$
d \bar X_t =  (b - a  \bar X_t   ) dt + \sigma  \sqrt{  f (\bar X_t ) } d W_t  +  \int_{ \R  \times \R_+} \frac{d}{(1 + z)^2 }  1_{ \{ u  \le  f( \bar X_{t-}) \}}  N (dt,d z, du ) .
$$
Taking the Lyapunov function $V(x) = |x|^2, $ it is easy to see that $\bar X_t $ satisfies all conditions of Assumption \ref{ass:Final!} for any choice of $x_0 \in K $ since the diffusion part is uniformly elliptic (recall that $ f $ is strictly lower bounded).  Concerning the time inhomogeneous process, Assumptions \ref{conditions} and \ref{ass:1} are satisfied by construction. To check \eqref{eq:lyapunovlt}, it suffices to take once more $ V(x) = |x|^2 $ and $t_0 $ sufficiently large (such that $ a e^{-2 rt_0} < 2$). Finally, it is straightforward to verify that $C_{t_0} $ defined in \eqref{eq:ct} is finite. 

Therefore, Theorem \ref{theo:2} and Corollary \ref{cor:3} apply in this case.

\appendix

\section{Proof of Proposition \ref{prop:goodcontrol}}\label{sec:App}

The proof of Proposition \ref{prop:goodcontrol} follows a well-known scheme that we briefly sketch now. In the following, all assumptions of Proposition \ref{prop:goodcontrol} are supposed to be satisfied. Let $ C' = \{ x : |x- x_0| < \eta / 2 \} .$ 

\begin{cor}
Under Assumption \ref{ass:strongcontrol} or \ref{ass:weakcontrol}, the following holds true. 

(i) There exists $ \varepsilon > 0 $ such that for all $ x, y \in K$ there exists a coupling of $ \bar X_t^x $ and $ \bar X_t^y $ with
$$ P_{x, y} ( (\bar X^x_1, \bar X^y_1) \in C' ) > \varepsilon > 0 .$$ 

(ii) Suppose in addition that Assumption \ref{ass:drift} holds. Then 
$ \tau_c (x, y ) = \inf \{ t : (\bar X_t^x , \bar X_t^y ) \in C' \times C' \} , $ for $ x, y \in K, $ 
possesses polynomial moments of all orders, i.e., for all $ p \geq 1,$
$$ E_{x, y } \tau_c (x, y)^p < \infty $$
 
\end{cor}

\begin{proof}
We use the independent coupling of $ \bar X_t^x$ and $\bar X_t^y, $ Corollary \ref{cor:412} and Theorem \ref{theo:controlXbar} to show that 
$$ \inf_{x, y \in K} P_{x, y } ( |\bar X_1^x - x_0 | \le \eta/2  , |\bar X_1^y - x_0 | \le \eta/2 ) := c_1 > 0 .$$ 
Introduce then the following sequence of stopping times. 
$$s_1 = \inf \{ t \geq 0 : (\bar X_t^x, \bar X_t^y ) \in K \times K \}, s_{n+1} = \inf \{ t \geq s_n + 1 : (\bar X_t^x, \bar X_t^y ) \in K \times K \} , $$
for all $n \geq 1 .$ Finally, put 
$$ \tau^* = \inf \{ n : (\bar X_{s_n +1}^x , \bar X_{s_n +1}^y) \in   C' \times C' \} ,$$
then clearly 
$$ P (\tau^* > n ) \le (1- c_1)^n .$$ 
Notice that we have 
$ \tau_c (x, y) \le S_{\tau^*}+ 1 $
and that $ s_{n+1} =   s_n + \tau_{ K \times K }(1) \circ \vartheta_{s_n} , $ where 
$$ \tau_{K \times K } ( \delta ) = \inf \{ t \geq \delta : (\bar X_t^x, \bar X_t^y) \in K \times K \} ,$$
for which we have the control 
$$ E_{x, y } (e^{b \tau_{K \times K } ( \delta )} ) \le V(x ) + V(y ) + C (\delta ), $$
thanks to Theorem 4.1 of \cite{DFG09}. This implies in particular that 
$$ E_{x, y }  (s_{n+1} - s_n)^p \le C (p ) $$
for a constant not depending on $x, $ nor on $y, $ for all $ n \geq 1.$ Then the main estimate which allows to conclude is 
$$ E_{x, y } (s_{\tau^*})^p \le E_{x, y } s_1^p + 2^p \sum_{n\geq 1  } \left[ E_{x, y }  (s_{n-1})^p + C(p ) \right]  (1- c_1)^{n- 1 } \le V(x) + V( y ) + \tilde C(p ) < \infty .$$
\end{proof}

The main idea is now to show that once the continuous time process has reached $ C'  \times C' , $ it takes some time to exit from $ C \times C .$ Taking $ \Gamma_n, $ the rate of the dominating Poisson process, sufficiently large, the probability that a jump $T_k^{[n]} $ arises during this time, i.e.\ before exiting from $C,$ can then be made arbitrarily large. Following arguments of \cite{Victor}, we therefore obtain the following result.

\begin{prop}\label{prop:33}[Proposition 3.3 of \cite{Victor}]
Grant Assumptions \ref{conditionsbis} and \ref{conditions2bis}. Then there exists $n_0, $ such that for any  $n \geq n_0,$ 
\begin{equation}\label{eq:lowerbound}
 \inf_{ x, y  \in C' } P_{x, y }  ( \bar X^x_{T^{[n]}_1 - }, \bar X^y_{T^{[n]}_1 - }) \in C \times C  )  \geq \frac{1}{2} .
\end{equation}
\end{prop}

The assertion of Proposition \ref{prop:goodcontrol} follows then along the lines of the proof of Proposition 3.6 of \cite{Victor}.

\section{Proof of Theorem \ref{cor:1}}
To prove Theorem \ref{cor:1}, we start by recalling the following result of \cite{ballyetc}.

\begin{theo}[Theorem 14 of \cite{ballyetc}]\label{theo:diff}
Suppose that the coefficients of \eqref{eq:sde0} satisfy Assumption \ref{conditionsbis}. Then there exists a constant $C > 0$ such that for any $ f \in C^q_b ( \R^d ) ,$ for any $ T > 0 $ and for all $t \le T, $ 
$$  \| P_t f \|_{q, \infty } \le C \|f\|_{q, \infty } ( t \vee 1 )  \alpha^{2q  }_{q, 4 q  } ( C , T ) \times \left( 1+ \Gamma_{ q} ( \gamma) + \sum_{ 0 \le |\beta | \le q } [ \partial^\beta_x \ln \gamma ]_{ 4q } \right)^q .$$
\end{theo}
Observe further that  \eqref{eq:ct} implies that 
$$ \| \nabla L_t f \|_\infty \le C_{t_0} \|f\|_{3, \infty } ,$$
for all $ t \geq t_0 ,$ which is condition (43) of \cite{ballyetc} with $k=1$ and $ q =3 .$ 

As a consequence, all conditions needed to obtain Theorem 16 of \cite{ballyetc} are satisfied and we obtain

\begin{theo}\label{theo:1}[Theorem 16 of \cite{ballyetc}]
Suppose that the coefficients of \eqref{eq:sde} and of \eqref{eq:sde0} satisfy Assumption \ref{conditions} and \ref{conditionsbis}. Suppose moreover that $ Q_3 (P, T ) < \infty $ for all $ T > 0 $ and that $C_{t_0 } < \infty $ for some $t_0 > 0 .$ Finally, grant Assumptions \ref{ass:1} and  \ref{unifbounded}.

Then there exists a constant $C $ not depending on $t_0, $ such that for all $ t_0 \le  t \le t_0 +  T ,$
\begin{equation}
| E ( f ( X_{t_0, t }^x)) - E (f ( \bar X^x_{t_0, t} ) ) | \le C  Q_3(P, T ) \psi ( x) \; \| f \|_{3, \infty}  \; [( t-t_0)\vee 1]  \int_{t_0}^t  e^{-rs}  ds  ,
\end{equation}
where $ \psi ( x) = 1 + |x| .$ 
\end{theo} 

We are now able to conclude the 
\begin{proof}[Proof of Theorem \ref{cor:1}]
Theorem \ref{theo:1} implies that for all $s \le T,$ 
$$ d_{\mathcal F} ( \mu_{t + s } , \mu_t P_s ) \le C Q_3 (P, T) [\int \psi (x) \mu_t (dx)]  T \int_t^{t+T} e^{-rs} ds .$$
By definition of $ Q_3 ( P, T ), $ we have that $ Q_3 (P, T ) \le C e^{ M_1 T  },$ for some suitable constant $M_1.$ Moreover,  Assumption \ref{unifbounded} implies that $\int \psi (x) \mu_t (dx) = 1 + E | X_t| < C(1 + E |X_0 | ) < \infty .$ This concludes our proof.
\end{proof}
\section*{Acknowledgements}
The author would like to thank Julien Chevallier for careful reading and useful remarks.

\end{document}